%% file: Monochrom-subgraph-arxiv.tex
\documentclass[aap,reqno]{imsart}

\RequirePackage{amsthm,amsmath,amsfonts,amssymb}
\RequirePackage[numbers]{natbib}
\RequirePackage{color}
\RequirePackage{mathtools,enumitem}

\RequirePackage{csquotes}

\RequirePackage{subcaption}

\RequirePackage[colorlinks,citecolor=blue,urlcolor=blue]{hyperref}
\RequirePackage{graphicx}
\input{preamble}

\usepackage{svg}
\usepackage{bm}
\usepackage{cleveref}
\usepackage{tikz}
\usepackage{ushort}

\startlocaldefs

\theoremstyle{remark}

\makeatletter

\newcommand{\customitem}[1]{%
\item[\rm#1]\protected@edef\@currentlabel{#1}%
}
\makeatother

\newcommand\rou[1]{\lfloor{#1}\rfloor}
\newcommand\ceil[1]{\lceil{#1}\rceil}
\newcommand{\zhu}[1]{{\color{blue}[#1]}}
\newcommand\bc[1]{\left({#1}\right)}
\newcommand\cbc[1]{\left\{{#1}\right\}}

\newcommand\brk[1]{\left\lbrack{#1}\right\rbrack}

\newcommand\abs[1]{\left|{#1}\right|}
\usepackage{dsfont}

\newcommand\Erw{\mathbb{E}}
\newcommand\ZZ{\mathbb{Z}}

\newcommand\PP{\mathbb{P}}

\newcommand{\dtv}{\mathrm{d}_{\mathrm{TV}}}

\newcommand{\longversion}[1]{}

\newcommand{\rd}{{\rm d}}
\newcommand{\aut}{{\rm aut}}
\endlocaldefs

\newcommand{\Max}{{\rm MaxC}}
\newcommand{\Min}{{\rm MinM}}

\newcommand{\dg}[1]{{\small \color{brown}[David: #1]}}

\newcommand{\yd}[1]{{\small \color{blue}[Yatin: #1]}}

\newcommand{\G}{\mathbb{G}}
\newcommand{\Hg}{\mathbb{H}}

\newcommand{\ignore}[1]{\relax}

\begin{document}

\begin{frontmatter}
\title{Minimum Number of Monochromatic Subgraphs of a Random Graph}

\begin{aug}
\author{\fnms{Yatin} \snm{Dandi}\ead[label=e1]{yatin.dandi@epfl.ch}}
\and
\author{\fnms{David} \snm{Gamarnik}\ead[label=e2]{gamarnik@mit.edu}}
\and
\author{\fnms{Haodong} \snm{Zhu}\ead[label=e3]{h.zhu1@tue.nl}}
\address{Statistical Physics of Computation Laboratory, École Polytechnique Fédérale de Lausanne (EPFL), \\\printead{e1}}
\address{Operations Research Center and Sloan School of Management, MIT, Cambridge, \\\printead{e2}}
\address{Department of Mathematics and Computer Science, Eindhoven University of Technology, \\\printead{e3}}

\end{aug}

\begin{abstract}
We consider the problem of minimizing the number of monochromatic subgraphs of a random graph, when each node of the host graph is
assigned one of the two colors. Using a recently discovered contiguity between appearance of 
strictly balanced subgraphs $F$ in a random graph, and random hypergraphs where copies 
of $F$ are generated independently, we show that the minimum value converges to a limit,
when the expected number of copies of $F$ is 
linear in the number of nodes $|V|$.
Furthermore, using the connections with mean field
spin glass models, we
obtain an asymptotic expression for 
this limit as the normalized expected number of copies
of $F$ 
and the size of $F$ diverge to infinity.
\end{abstract}



\end{frontmatter}

\section{Introduction}
The study of monochromatic subgraphs in graph bipartitions is a widely investigated problem in extremal combinatorics and Ramsey theory. 
A classical result by Goodman~\cite{g1959} states that in any bipartition of the complete graph $K_n$, the number of monochromatic triangles admits a sharp lower bound
\[
T(n) = 
\begin{cases}
    \dfrac{u(u-1)(u-2)}{3}, & \text{if } n = 2u; \\
    \dfrac{2u(u-1)(4u+1)}{3}, & \text{if } n = 4u+1; \\
    \dfrac{2u(u+1)(4u-1)}{3}, & \text{if } n = 4u+3.
\end{cases}
\]
\ignore{
\zhu{The general case was studied by Paul Erdős, who proved that the minimum number of monochromatic $K_r$ in any bipartition of $K_n$ is at most $\binom{n}{r}/2^{\binom{r}{2}-1}$ and conjectured that this bound is asymptotically tight for large $n$ \cite{erdHos1962number}. This conjecture was later disproved by Thomason \cite{thomason1989disproof}. In 2012, Conlon proved that any bipartition of $K_n$ contains at least $n^r / C^{(1+o(1))r^2}$ monochromatic $K_r$ subgraphs, where $C \approx 2.18$ \cite{conlon2012ramsey}. This phenomenon, whereby a positive fraction of all $r$-cliques are monochromatic, also appears in other deterministic contexts \cite{frankl1988quantitative}. In general, it can often be characterized by the principle that monochromatic substructures are evenly distributed among configurations that force their occurrence \cite{csv}.}
\dg{I am not sure we need such a detailed summary about deterministic setup. It is somewhat off-topic}
}

While these deterministic settings have attracted considerable attention and yielded several intriguing results, 
the same question in the context of random graphs,
such as Erdős–Rényi random graphs, remains far less understood, even for  the special  case of triangles. In this paper, we consider this question of finding the smallest number of monochromatic subgraphs in bipartitions of random graphs. We do so in a general setting, 
where the subgraph is only assumed to be
strictly balanced.

To set the stage, let $F$ be a graph on $r$ vertices. An $F$-graph $H$   is defined as any pair $(V=V(H), E=E(H))$, where $V$ is a set of vertices, and 
$E$ is a collection of copies $e$ of $F$ called hyperedges, where for each $e$, the vertex set of $e$ is subset of $V$.
 A random \( F \)-graph \( \Hg_F(n, q) \) is an \( F \)-graph with vertex set \( V=[n]=\cbc{1,2,\ldots,n} \), where each of the
\begin{align}\label{eq:N-trials}
\binom{n}{r} \cdot \frac{r!}{\aut(F)}
\end{align}
potential copies of \( F \) on vertices in \( [n] \) is included in the edge set $E$ independently with probability \( q \). Here $\aut(H)$ denotes the set of automorphisms of a graph $H$. In the special case
when $F$ is $K_2$ (two nodes
connected by an edge), we use a more common notation $\G(n,q)$.

Given a simple graph $G=(V,E)$ (so that $F$ is just $K_2$) and given its bipartition \(  V=V_1 \cup V_2, V_1\cap V_2=\emptyset \), a cut, or more specifically an $F$-cut associated with this bipartition is defined as the set of subgraphs $e$ of $G$ which are isomorphic to $F$, and  whose vertex sets have a non-empty intersection with  \( V_1 \) and \( V_2 \). That is, each hyperedge of the cut is  neither entirely contained in \( V_1 \), nor is it entirely contained in \( V_2 \). The  \emph{cut value} associated with a cut is 
the cardinality of this set.
 The \emph{max-cut value}
 denoted by ${\rm MaxC}(G,F)$
 is defined as the largest such value taken over all possible bipartitions of \(V\). Thinking of $V_1$ and $V_2$ as being associated with two distinct colors, the cut associated with the bipartition is simply the set of all non-monochromatic subgraphs of $G$ isomorphic to $F$
  associated with the coloring scheme $V_1,V_2$. Conversely, every $F$-subgraph of $G$ not participating in the cut is monochromatic by definition. The minimum
  number of monochromatic $F$-subgraphs is denoted by ${\rm MinM}(G,F)$. Naturally, ${\rm MinM}(G,F)+{\rm MaxC}(G,F)$ is the total number of $F$-subgraphs of $G$ as this is the case for every bipartition.

In combinatorial optimization and theoretical computer science, there is a long-standing interest in understanding the    minimum number of monochromatic subgraphs achievable by  varying over all bipartitions of  graphs, both in the special case of max-cut values associated with $F=K_2$, and beyond ~\cite{ckpsty2013,dms2017,dl2018,sen2018optimization,shabanov2021maximum}. This is the question we address in our paper.

\ignore{
Indeed, there is a close connection between these two quantities. For a given graph, its max-cut plus the minimum number of monochromatic edges equals the total number of edges in the graph. However, when it comes to other monochromatic subgraph, like triangles, the relation seems to be not so apparent. 
}

Let $F$ be a graph. We say $F$ is a strictly 1-balanced graph if $\rd_1(F')<\rd(F)$ for any subgraph $F'$ of $F$ such that $F'\neq F$.  Here
    \begin{align*}
        \rd_1(G)=\frac{\abs{E(G)}}{\abs{V(G )}-1}.
    \end{align*}
There are many examples of strictly 1-balanced graphs, including  cycles and complete graphs. However, any disconnected graph is not strictly 1-balanced, and neither is a tree-graph.

Recently a great progress was achieved in understanding the relationship between $\G(n,p)$ and $\Hg_F(n,q)$
with $q$ judiciously chosen
as $q=p^{|E(F)|}$. In particular, as shown in \cite{burghart2024sharp} (see Proposition~\ref{prop:coupling} below), when $p$ is at most $n^{-1/d_1(F)}$, the graphs
$\G(n,p)$ and $\Hg(n,q)$ can be coupled in such a way that the number of $F$-hyperedges in $\Hg$ nearly matches the number of copies of $F$ naturally occurring in $\G(n,p)$. As a result a difficult problem of studying minimal number of monochromatic $F$-subgraphs of $\G(n,p)$ can be reduced to simpler version, one defined on $\Hg$ where the occurances of $F$-s are independent by design. This is the insight we use to obtain several asymptotic results on max-cut values for general graphs $F$ in $\G(n,p)$. 

Our first result is as follows.

\begin{theorem}\label{thm-convergence}
    Let $F$ be a strictly 1-balanced graph with $s>0$ edges on $r \geq 2$ vertices. Given $c>0$, let $p= cn^{-1/\rd_1(F)}$ and $q=p^s=c^sn^{1-r}$.   There exists a constant $m(F,c)$ such that
    \begin{align*}
        \frac{{\rm MinM}(\G(n,p),F)}{n}
        \to m(F,c),
    \end{align*}
\whp~ as $n\to\infty$.
\end{theorem}

\ignore{
\yd{Maybe we should specify that $m(F,p)$ is a deterministic scalar.}\dg{good point. Done}
}

Our proof approach is based on similar results 
for the contiguous model $\Hg(n,q), q=p^s$ which are already known in the literature. Specifically, the existence of $m(F,c)$ such that 
\begin{align*}
\frac{{\rm MinM}(\Hg(n,q),F)}{n}
\to m(F,c),
\end{align*}
is already known~\cite{shabanov2021maximum} based on the combinatorial interpolation technique introduced in~\cite{bayati2010combinatorial} and used
for establishing the existence of
such limits.
For similar contiguity reasons 
we also have that  
$m(F,c)=0$ when $c$ is 
sufficiently small positive 
constant $c$, and $m(F,c)>0$ when $c$ is sufficiently large. 
The former claim  is
obtained by choosing $c$ small enough so that
the  random graph $\Hg(n,q)$ does not contain
a giant component. In this case the 
locally tree-like
structure of the graph allows for cutting nearly
every $F$-edge  of the graph. Conversely, when
$c$ is sufficiently large, the 
non-existence of near
perfect bipartition follows from a simple union
bound. We regard these observations as folklore 
and will not provide a formal verification of 
these claims.

We conjecture that the value 
${\rm MinM}(\G(n,p),F){n}$ also undergoes a different type of phase transition. Specifically, we conjecture the existence of $c^*$ which 
depends on $F$ only such that this value is $0$ w.h.p. as $n\to\infty$
when $c<c^*$, and is positive  when $c>c^*$ also w.h.p.
The basis for this conjecture is a similar conjecture for random K-SAT (and several other related models such as proper coloring
of a random graph), which was proven for large $K$~\cite{ding2022proof}, but is still open for general $K$.

Our next result concerns obtaining explicit 
asymptotic limit
values when the size of the host graph $F$ grows.  
Introduce a short-hand notation
\begin{align*}
\kappa(F,c,s,r)={c^s\over 2^{r-1}\aut(F)}.
\end{align*}

\begin{theorem}\label{thm-limit}
   Under the assumption of \Cref{thm-convergence},
\begin{align*}
       m(F,c)=\kappa(F,c,s,r)+
    (1+o_r(1))\sqrt{(2\log 2)\kappa(F,c,s,r)}
       +o_{c}(c^{s/2}).
   \end{align*}
\end{theorem} 
Namely, we obtain explicit limit values when both the number of nodes of the host graph $F$ and the leading coefficient $c$ of the random graph 
parameter diverge to infinity. The proof inspiration
is based on large-$p$
approximation technique
for $p$-spin glass models employed recently
in~\cite{gamarnik2025shattering}, and porting these
results to sparse random graphs as was done 
in~\cite{dms2017} 
and~\cite{sen2018optimization}. The latter two works relate $p$-spin models to sparse graphs using the Lindeberg's interpolation method.
Instead, we develop relevant asymptotics directly for sparse graphs using the second 
moment method employed 
in~\cite{gamarnik2025shattering}
for the mean field $p$-spin model.

\paragraph{Organization} The remainder of this paper is organized as follows. In \Cref{sec-hypergraph-coupling}, we introduce a coupling between the random $F$-graph and the Erdős–Rényi random graph, and prove \Cref{thm-convergence}. In \Cref{sec-spinglass}, we reduce the study of \( m(F,p) \) to the analysis of the maximum of a family of Gaussian random variables. Section~\ref{sec-maximizer}  provides auxiliary results that en route to deriving this maximum. Finally, we prove \Cref{thm-limit} in \Cref{sec-proof}.

 \section{Limit existence}\label{sec-hypergraph-coupling}
In this section, we prove \Cref{thm-convergence}. 
To count the number of specific subgraphs in a graph, we introduce the notion of an \emph{ordered copy} as follows: 
\begin{definition}[Ordered copy]
    Let \( F \) be a graph with vertex set \( V(F) = \{f_1, \ldots, f_r\} \). 
For any tuple of vertices \( (v_1, \ldots, v_r) \), we define the \emph{ordered copy} 
\( F(v_1, \ldots, v_r) \) to be a graph on the vertex set \( \{v_1, \ldots, v_r\} \), 
where for all \( i, j \in [r] \) an edge is placed between \( v_i \) and \( v_j \) if and only if there is an 
edge between \( f_i \) and \( f_j \) in \( F \). 
\ignore{\yd{I think the order of quantifiers should be: for any \( i, j \in [r] \), an edge is placed between \( v_i \) and \( v_j \) if and only if there is an 
edge between \( f_i \) and \( f_j \) in \( F \)}
\dg{done}
}
\end{definition}

\ignore{
\yd{Minor point but here we are conflating the mapping from ordered tuples \( \{v_1, \ldots, v_r\} \) to the subgraph with the subgraphs. Maybe we can first define the mapping from tuples to graphs and then define the set of ordered copies as the range of this mapping}\dg{good point. I made a change below}
}
We emphasize that even if the tuples \( (v_1, \ldots, v_r) \) and \( (v_1', \ldots, v_r') \) are different, it is possible that \( F(v_1, \ldots, v_r)\) and  \(F(v_1', \ldots, v_r') \) are isomorphic, in which case we do not distinguish between them. In particular, in subgraph counting, such copies will be counted only once.

Let us give an example.
\begin{example}\label{exa_1}
Consider the graphs \( G \) and $F$ below:
\begin{center}
\begin{tikzpicture}[scale=2, every node/.style={circle, draw, fill=white, inner sep=1pt}]

  \begin{scope}
    \node (g1) at (0,1) {1};
    \node (g2) at (1,1) {2};
    \node (g3) at (1,0) {3};
    \node (g4) at (0,0) {4};

    \draw (g1) -- (g2);
    \draw (g2) -- (g3);
    \draw (g3) -- (g4);
    \draw (g4) -- (g1);
    \draw (g1) -- (g3);
    \draw (g2) -- (g4);

    \node[draw=none, fill=none] at (0.5, -0.5) {\textbf{Graph \( G \)}};
  \end{scope}

  \begin{scope}[xshift=2.8cm]
    \node (f1) at (0,1) {1};
    \node (f2) at (1,1) {2};
    \node (f3) at (1,0) {3};
    \node (f4) at (0,0) {4};

    \draw (f1) -- (f2);
    \draw (f2) -- (f3);
    \draw (f3) -- (f4);
    \draw (f4) -- (f1);
    \draw (f1) -- (f3);

    \node[draw=none, fill=none] at (0.5, -0.5) {\textbf{Subgraph \( F \)}};
  \end{scope}

\end{tikzpicture}
\end{center}
The graph \( G \) contains exactly two ordered copies of \( F \), since \( F(1,2,3,4) \) remains the same under the exchange of positions between \( 1 \) and \( 3 \), or between \( 2 \) and \( 4 \), whereas \( F(1,2,3,4) \) and \( F(2,3,4,1) \) represent different ordered copies in $G$. 
\end{example}

The following key result from \cite{burghart2024sharp} helps us understand the ordered copies of a strictly 1-balanced graph $F$ in $\G(n,p)$, and establishes a connection to hyperedges in the random $F$-graph $\Hg_F(n,q)$:
\begin{proposition}[{\cite[Theorem 1.6]{burghart2024sharp}}]\label{thm-bur-1.6}\label{prop:coupling}
Let $F$ be a strictly 1-balanced graph with $s > 0$ edges on $r \geq 2$ vertices. There exist constants $\delta, \varepsilon > 0$ such that the following holds. Fix any two sequences $p,q$ satisfying $p \leq n^{-1/\rd_1(F) + \varepsilon}$ and $q \leq (1 - n^{-\delta})p^{s}$. Whp as $n\to\infty$ there exists a permutation $\tau: [n] \to [n]$ such that 
 for every $F(v_1, \ldots, v_r) \in E(\Hg_F(n,q))$ also $F(\tau(v_1), \ldots, \tau(v_r)) \subset \G(n,p)$.
\end{proposition}
\ignore{
\yd{I think the order should be "whp as $n\to\infty$, here exists a permutation $\tau: [n] \to [n]$". The permutations aren't fixed beforehand.}
\dg{done}
\yd{Can't $\epsilon, \delta$ be arbitrarily small?}
\dg{We don't want them to be small. In fact the larger is $\epsilon$ the stronger the result. $\epsilon$ controls how high above $n^{-{1\over d_1(F)}}$ we can go and theorem still be valid. Makes sense?}
\yd{}
}
In the above $A\subset B$ means $A$ is a subgraph of $B$ (not necessarily induced one). 
\Cref{thm-bur-1.6} implies  that if we choose $q = (1 - n^{-\delta})p^{s}$, each $F$-hyperedge in $\Hg_F(n,q)$ corresponds to an ordered copy of $F$ in $\G(n,p)$. If we instead choose $q = p^s$, then, on the one hand, the number of $F$-hyperedges in $\Hg_F(n,q)$ does not increase too much, while on the other hand, it is close to the expected number of ordered copies of $F$ in $\G(n,p)$. In particular,
we claim the following.

\begin{corollary}\label{cor-erg-hg}
Let $F$ be a strictly 1-balanced graph with $s > 0$ edges on $r \geq 2$ vertices. Fix $c > 0$, and let $p = c n^{-1/\rd_1(F)}$,  $q = p^s = c^s n^{1-r}$.  Whp as $n\to\infty$, there exists a permutation
$\tau:[n]\to [n]$ such that the number of distinct ordered copies $F(v_1, \ldots, v_r)$ for which exactly one of the events $F(v_1, \ldots, v_r) \in E(\Hg_F(n,q))$ or $F(\tau(v_1), \ldots, \tau(v_r)) \subset \G(n,p)$ holds, is  $o(n)$.
\end{corollary}
\ignore{
\yd{Again the existence of permutation is whp}
\dg{done}
}
The proof crucially uses the fact that
for strictly 1-balanced graph  that when $p=O(n^{-{1/d_1(F)}})$, the number of pairs of distinct ordered copies of $F$ in $\G(n,p)$ which share at least one edge is small. 
\ignore{
\yd{remove "ensures that"}\dg{done}
}

Specifically, cases like \Cref{exa_1} occur only rarely.
 This property is essential for the validity of \Cref{thm-bur-1.6} and \Cref{cor-erg-hg}.

\begin{proof}[Proof of Corollary~\ref{cor-erg-hg}]
\ignore{\yd{Let $\delta>0$  be as in Proposition~\ref{prop:coupling}.}\dg{done}}
Let $\delta>0$  be as in Proposition~\ref{prop:coupling}.
We view $H'\triangleq \Hg_F(n,(1-n^{-\delta})q)$ 
as a subgraph of $\Hg_F(n,q)$. 
\ignore{\yd{Typo: should be $n^{-\delta}$}\dg{done}}

Let
\[
g(n) = \binom{n}{r} \frac{r!}{\aut(F)} p^s = \binom{n}{r} \frac{r!}{\aut(F)} c^s n^{1 - r} = \Theta(n)
\]
denote the expected number of $F$-edges in 
\( \Hg_F(n,q) \).  Let $\tau$ be the permutation
for which the event in 
Proposition~\ref{prop:coupling} holds \whp~. Let \( \mathcal{F} \) denote the event that, 
for each \( F(v_1, \ldots, v_r) \in E(H') \), 
we have \( F(\tau(v_1), \ldots, \tau(v_r)) 
\subset \G(n,p) \), and that
\[
\abs{E(H')} \geq g(n)(1 - n^{-\delta}) - n^{2/3}.
\]
Then, \Cref{thm-bur-1.6} yields that
\begin{align}\label{eq-Sc-decom}
    \PP\left(\mathcal{F}^c\right) \leq o(1) + \PP\left( \abs{E(H')} \leq g(n)(1 - n^{-\delta}) - n^{2/3} \right).
\end{align}
We have that $E(H')$ has a binomial
distribution with the number of trials
given by (\ref{eq:N-trials}) and success
probability $(1-n^{-\delta})q$. The mean and variance are $\Theta(n)$ each, and thus 
by Chebyshev's inequality the probability term
in \eqref{eq-Sc-decom} is $o(1)$ as well.
We conclude $\PP(\mathcal{F}^c)=o(1)$.

Let $N(\G,F)$ be the number of subgraphs 
of $\G(n,p)$ isomorphic to $F$. We have
$\Erw[N(\G,F)]=g(n)=\Theta(n)$ and
the variance also $\Theta(n)$.  Thus, similarly, by Chebyshev's bound
\whp~ $N(\G,F)\le g(n)+n^{2\over 3}$.
Combining this with $\PP(\mathcal{F})=1-o(1)$, and $g(n)n^{-\delta}=o_n(n)$, we obtain the result.
\end{proof}
\ignore{\zhu{I have a question about notation: should we write $o(1)$ or $o_n(1)$ }
\dg{Good point. I have changed to $o_n(n)$}}

We  now prove \Cref{thm-convergence}.
\begin{proof}[Proof of \Cref{thm-convergence}]
We invoke \cite[Theorem 1]{shabanov2021maximum}
which states the following. Consider a random $r$-uniform hypergraph $\Hg_{K_r}(n,\bar q)$ where
$\bar q=\bar cn^{1-r}$ for some fixed $\bar c>0$. There exists
$m(r,\bar c)$ such that
\begin{align}
{{\rm MinM}(\Hg_{K_r}(n,\bar q),K_r)\over n}
\to m(r,\bar c), \label{eq:limit-K_r}
\end{align}
\whp~ as $n\to\infty$. In other words, the claim
of the theorem holds for $F=K_r$.

We set $\bar c={r!/\aut(F)}c$ and with corresponding $\bar q=\bar cn^{1-r}$ and consider a natural coupling of $\Hg_F(n,q)$ 
and $\Hg_{K_r}(n,\bar q)$ constructed as follows.
Given a sample of $\Hg_F(n,q)$ for each $r$-subset
$B$ of $[n]$ we create an $K_r$ edge of out of $B$ iff at least one 
of $r!/\aut(F)$ copies of $F$ is supported on $B$
in $\Hg_F$. Clearly appearance of hyperedges
$K_r$ in this construction is independent across
different sets $B$. Also
we have that for each fixed $B$, the number of 
such copies is $1$ with probability 
$(r!/\aut(F))q(1+o(1))=\Theta(n^{1-r})$, 
is at least two with probability $O(n^{2-2r})$
and is zero with probability $1-(r!/\aut(F))q(1+o(1))$. Thus we obtain a random
graph $\Hg_{K_r}(n,\tilde q)$ where 
$\tilde q=\bar q(1+o(1))$.
We can further couple
this graph with $\Hg_{K_r}(n,\bar q)$
by adding or deleting copies of $K_r$. The
 number of such copies is 
$o(n)$ whp as $n\to\infty$, since $\tilde q-\bar q=o(1)$.
\ignore{\yd{Use of whp here unclear}\dg{changed the wording above. Clear now?}}
By our construction, \whp~,
\begin{align*}
{\rm MinM}(\Hg_F(n,q))=
{\rm MinM}(\Hg_{K_r}(n,\tilde q))+o(n)=
{\rm MinM}(\Hg_{K_r}(n,\bar q))+o(n).
\end{align*}
Applying (\ref{eq:limit-K_r}), we obtain
the result by defining  $m(F,c)=m(r,\bar c)$.
\end{proof}

\section{Reduction to spin glass model}\label{sec-spinglass}
In this section reformulate the problem
of minimizing the number of monochromatic subgraphs in terms of an associated version 
of a spin glass problem.
 
Given \( p = c n^{-1/\rd_1(F)} \), \( q = p^s = c^s n^{1 - r} \), $\G(n,p)$, and \(\Hg_F(n, q) \), we assign to each vertex \( i \in [n] \) a spin \( \sigma_i \in \{1, -1\} \), thereby inducing a bicoloring of the vertex set \( [n] \). Let \( \sigma = (\sigma_i)_{i \in [n]} \). Define \( H(\sigma) \) to be the number of hyperedges \( e \in E(\Hg_F(n,q)) \) such that all vertices in \( e \) receive the same spin value \ignore{\yd{spin value}
\dg{done}}. Namely, for all
$e$ contributing to $H(\sigma)$ \ignore{\yd{for all $e$?}} (viewed as $r$-subset of $[n]$), 
and all $i,j\in e$, it holds $\sigma_i=\sigma_j$.
In particular, $\min_\sigma H(\sigma)=
{\rm MinM}(\Hg_F(n,q))$.

Let $\mathcal{S}[k]$ denote the group of permutation on $k$ elements.

\begin{lemma}\label{lem:spin-form-H}
For any $\sigma\in\cbc{-1,1}^n$, 
\begin{align*}
H(\sigma) 
=& \frac{1}{\aut(F) \cdot 2^{r-1} r!} \sum_{v_1 \neq \cdots \neq v_r} \sum_{\rho \in \mathcal{S}[r]} \indic{F(v_{\rho(1)}, \ldots, v_{\rho(r)}) \in E(\Hg_F(n,q))} 
\left( \sum_{m = 1}^{\lfloor r/2 \rfloor} \sum_{\substack{R \subset [r] \\ |R| = 2m}} \prod_{j \in R} \sigma_{v_j} \right)\\
&+ \frac{|E(\Hg_F(n,q))|}{2^{r-1}}.
\end{align*}
\end{lemma}
\begin{proof}
We have
\begin{align*}
H(\sigma) 
= \sum_{e \in E(\Hg_F)} \left[ \indic{\sigma_i = 1 \;\; \forall i \in e} + \indic{\sigma_i = -1 \;\; \forall i \in e} \right].
\end{align*}
Since each hyperedge corresponds to an ordered copy of \( F \), we can express this sum as
\begin{align*}
H(\sigma) 
= \frac{1}{\aut(F)} \sum_{v_1 \neq \cdots \neq v_r\in [n]} 
\indic{F(v_1, \ldots, v_r) \in E(\Hg_F)} \left[ \indic{\sigma_{v_j} = 1 \;\; \forall j \in [r]} + \indic{\sigma_{v_j} = -1 \;\; \forall j \in [r]} \right].
\end{align*}

Next, using  \( \sigma_i \in \{-1, 1\} \), the indicator functions can be rewritten in algebraic form:
\begin{align*}
\indic{\sigma_{v_j} = 1 \;\; \forall j \in [r]} = \prod_{j=1}^r \frac{1 + \sigma_{v_j}}{2}, \qquad
\indic{\sigma_{v_j} = -1 \;\; \forall j \in [r]} = \prod_{j=1}^r \frac{1 - \sigma_{v_j}}{2}.
\end{align*}

Substituting these expressions, we obtain
\begin{align*}
H(\sigma) 
= \frac{1}{\aut(F)} \sum_{v_1 \neq \cdots \neq v_r} \indic{F(v_1, \ldots, v_r) \in E(\Hg_F)} \left[ \prod_{j=1}^r \frac{1 + \sigma_{v_j}}{2} + \prod_{j=1}^r \frac{1 - \sigma_{v_j}}{2} \right].
\end{align*}

We now expand the sum of products as follows:
\[
\prod_{j=1}^r \frac{1 + \sigma_{v_j}}{2} + \prod_{j=1}^r \frac{1 - \sigma_{v_j}}{2}
= \frac{1}{2^{r-1}} \left( \sum_{m = 1}^{\lfloor r/2 \rfloor} \sum_{\substack{R \subset [r] \\ |R| = 2m}} \prod_{j \in R} \sigma_{v_j} \right) + \frac{1}{2^{r-1}}.
\]

We obtain
\begin{align*}
H(\sigma) 
= \frac{1}{\aut(F) \cdot 2^{r-1}} 
\sum_{v_1 \neq \cdots \neq v_r\in [n]} \indic{F(v_1, \ldots, v_r) \in E(\Hg_F)} \left( \sum_{m = 1}^{\lfloor r/2 \rfloor} \sum_{\substack{R \subset [r] \\ |R| = 2m}} \prod_{j \in R} \sigma_{v_j} \right)
+ \frac{|E(\Hg_F)|}{2^{r-1}}.
\end{align*}
Let \( \mathcal{S}[m] \) denote the set of all permutations on \( [m] \).  To symmetrize over labelings of \( F \), we average over permutations in \( \mathcal{S}[r] \). This yields:
\begin{align*}
H(\sigma) 
=& \frac{1}{\aut(F) \cdot 2^{r-1} r!} \sum_{v_1 \neq \cdots \neq v_r} \sum_{\rho \in \mathcal{S}[r]} \indic{F(v_{\rho(1)}, \ldots, v_{\rho(r)}) \in E(\Hg_F)} 
\left( \sum_{m = 1}^{\lfloor r/2 \rfloor} \sum_{\substack{R \subset [r] \\ |R| = 2m}} \prod_{j \in R} \sigma_{v_j} \right)\\
&+ \frac{|E(\Hg_F)|}{2^{r-1}},
\end{align*}
as claimed.
\end{proof}
The remainder of the proof is based on the following ideas. We replace the random variable \( \{ \sum_{\rho \in \mathcal{S}[r]} \indic{F(v_{\rho(1)}, \ldots, v_{\rho(r)}) \in E(\Hg_F)}  \}_{v_1 \neq \cdots \neq v_r} \) by  Gaussian random variables with matching mean and variance. 
\ignore{\yd{I'm not sure if the independence is clear here. Isn't the independence only upto symmetry?}\dg{deleted ''independence''} \yd{$E(R)$ or $E(\Hg_F)?$}
\dg{good catch. Fixed}}

Then \( H(\sigma) \)  itself is also Gaussian for each fixed configuration \( \sigma \).  

Since there are \( 2^n \) possible spin configurations \( \sigma \in \{-1, 1\}^n \), this  yields a family of \( 2^n \) correlated Gaussian random variables \( \{ H(\sigma) \}_{\sigma \in \{-1, 1\}^n} \).  
Next we will control the covariance matrix of these Gaussian random variables and the reduce the optimization problem to the variational problem over correlated Gaussians for which are able to obtain asymptotic values. The justification of the Gaussian substitution is done using Lindeberg's interpolation method when $c\to\infty$ as carried out in \cite{sen2018optimization}. Thus our next goal is to restate the problem in a setup accorded by~\cite{sen2018optimization}.

Let \( \kappa_1 = -\frac{1}{\aut(F)\cdot 2^{r-1}} \), \( \kappa_2 = \frac{1}{2^{2r - 2} r! \aut(F)} \), and let \( d = c^s \).
Define the symmetric polynomial
\begin{align}
f(x_1, \ldots, x_r) = \sum_{m = 1}^{\lfloor r/2 \rfloor} \sum_{\substack{R \subset [r] \\ |R| = 2m}} \prod_{j \in R} x_j.
\end{align}
Let \( \{ J_{v_1, \ldots, v_r}: v_1, \ldots, v_r \in [n] \}\) be a symmetric array of 
random variables where 
\[
\left\{ J_{v_1, \ldots, v_r} : 1 \leq v_1 \leq \cdots \leq v_r \leq n \right\}
\]
are independent standard normal variables. I.e., \( J_{v_1, \ldots, v_r} \sim \mathcal{N}(0, 1) \) independently over $1 \leq v_1 \leq \cdots \leq v_r \leq n$.
Define the Gaussian random variable 
\begin{align*}
U_n(\sigma) = \sqrt{r!} \sum_{1 \leq v_1 \leq \cdots \leq v_r \leq n} g(v_1, \ldots, v_r) \, J_{v_1, \ldots, v_r} \, f(\sigma_{v_1}, \ldots, \sigma_{v_r}),
\end{align*}
where the symmetry factor \( g(v_1, \ldots, v_r) \) is defined by
\begin{align*}
g(v_1, \ldots, v_r) = \sqrt{ \left| \left\{ (v_{\rho(1)}, \ldots, v_{\rho(r)}) : \rho \in \mathcal{S}_r \right\} \right| }.
\end{align*}
We note that $g$ is simply a function of the number of distinct elements among $v_1,\ldots,v_r$.
We note that the above definition of \( g(v_1, \ldots, v_r) \) ensures that for any function
$h:[n]^r\to \R$, it holds
\begin{align}\label{eq-g-sum-property}
    \sum_{1 \leq v_1 \leq \ldots \leq v_r \leq n} g^2(v_1, \ldots, v_r) h(v_1, \ldots, v_r)
    = \sum_{v_1, \ldots, v_r \in [n]} h(v_1, \ldots, v_r).
\end{align}

\ignore{
\yd{I would briefly explain here why the $\alpha$ part is separated out in terms of order of fluctuations. (It's like the magnetization while the rest is the spin-glass part). }
\dg{don't see a clean way to phrase this. The audience is mostly folks in combinatorics/random graphs. Phrases like magnetization better be explained. If you have
an idea how to say it, go for it.}}

Introduce the constrained variational problem\footnote{The term $\sqrt{r!}g(v_1,\ldots,v_r)$ appears because the summation is ordered (with $v_1 \leq \ldots \leq v_r$). In \cite[equation (1.3)]{sen2018optimization}, Sen defined $T_n^\alpha$ without this ordering condition (using $v_1 \neq \ldots \neq v_r$ instead):
\[
T_n^\alpha = \frac{1}{n^{(r+1)/2}} \max_{\sigma \in \{-1,1\}^n} \sum_{v_1 \neq \dots \neq v_r} J_{v_1, \dots, v_r} f(\sigma_{v_1}, \cdots, \sigma_{v_r}).
\]
Hence, comparing it with \eqref{def-T-alp}, we need an extra factor of $r!$ for each unordered tuple $v_1 \neq \ldots \neq v_r$. }
\begin{equation}\label{def-T-alp}
    \begin{aligned}
        T_n^\alpha = n^{-(r+1)/2} \sqrt{\kappa_2} \max_{\sigma \in \{-1,1\}^n} U_n(\sigma),\\ 
\text{subject to} \quad
n^{-r} \sum_{v_1, \ldots, v_r \in [n]} \kappa_1 f(\sigma_{v_1}, \ldots, \sigma_{v_r}) = \alpha.
    \end{aligned}
\end{equation}
For completeness, the optimal value is set to $-\infty$ if the constraint above is infeasible.

Finally, define
\begin{align}\label{eq:Vn}
V_n = \mathbb{E} \left[ \max_\alpha \left( \alpha c^s + T_n^\alpha c^{s/2} \right) \right].
\end{align}

We now reduce the value of \( \frac{1}{n} \min_{\sigma \in \{-1,1\}^n} H(\sigma) \) to the optimization problem (\ref{eq:Vn}). 

\begin{proposition}\label{pro:H-V}
\whp~ as $n\to\infty$
\[
\frac{1}{n} \min_{\sigma \in \cbc{-1,1}^n} H(\sigma) 
= -V_n + \frac{c^s}{2^{r-1} \aut(F)} + o(1) + o_{c}(c^{s/2}).
\]
\end{proposition}

\begin{proof}[Proof of \Cref{pro:H-V}]

Let 
\begin{align}
    A_{v_1,\ldots,v_r}=-\frac{1}{\aut(F)2^{r-1}r!}\sum_{\rho\in \mathcal{S}[r]}\indic{F(v_{\rho(1)},\ldots,v_{\rho(r)})\in E(\Hg_F)}.
\end{align}
Then
\begin{align}
    \Erw\brk{A_{v_1,\ldots,v_r}}=-\frac{1}{\aut(F)2^{r-1}}c^sn^{1-r}=\kappa_1c^s n^{1-r}.
\end{align}
By definition, for any $\rho_0\in \mathcal{S}[r]$ there are precisely $\aut(F)$ elements $\rho \in \mathcal{S}[r]$ such that $$F(v_{\rho(1)},\ldots,v_{\rho(r)})=F(v_{\rho_0(1)},v_{\rho_0(2)},\ldots,v_{\rho_0(r)}).$$ Hence,
\begin{align*}
    \Var\bc{A_{v_1,\ldots,v_r}}&=\frac{1}{(2^{r-1}r!)^2}\frac{r!}{\aut(F)}p^s(1-p^s)=\frac{1}{2^{2r-2}r!\aut(F)}(1-p^s)c^sn^{1-r}\\
    &= \frac{1}{2^{2r-2}r!\aut(F)}c^sn^{1-r}(1+o_n(1))=\kappa_2c^s n^{1-r}(1+o_n(1)).
\end{align*}
Let
\begin{align*}
    L_n&=\frac{1}{n}\max_{\sigma \in \cbc{-1,1}^n}-H(\sigma)+\frac{\abs{E(\Hg_F)}}{2^{r-1}n}\\
    &=\frac{1}{n}\max_{\sigma\in\cbc{-1,1}^n}\sum_{v_1\neq \ldots\neq v_r}A_{v_1,\ldots,v_r}f(\sigma_{v_1},\ldots,\sigma_{v_r}).
\end{align*}
Then, \whp~,
\begin{align*}
    \frac{1}{n}\min_{\sigma \in \cbc{-1,1}^n}H(\sigma)=-L_n+\frac{c^s}{2^{r-1}\aut(F)}+o_n(1).
\end{align*}
Introduce
\begin{align}
    \hat{U}_n(\sigma)=\sum_{v_1\neq \ldots\neq v_r}J_{v_1,\ldots,v_r}f(\sigma_{v_1},\ldots,\sigma_{v_r}),
\end{align}
and
\begin{align}\label{def-hT-alp}
    \hat{T}_n^\alpha=\max_{\sigma \in \cbc{-1,1}^n}n^{-(r+1)/2}\sqrt{\kappa_2}\hat{U}_n(\sigma),\\
    \mbox{subject to}~n^{-r}\sum_{v_1\neq \ldots\neq v_r}\kappa_1 f(\sigma_{v_1},\ldots,\sigma_{v_r})=\alpha.\nn
\end{align}
By \cite[Theorem 1.1]{sen2018optimization},
    \begin{align*}
        L_n=\Erw\brk{\max_{\alpha}(\alpha c^s+\hat{T}_n^\alpha c^{s/2})}+o_{c}(c^{s/2}).
    \end{align*}
Hence, it remains to show that
\begin{align*}
    \Erw\brk{\max_{\alpha} \left( \alpha c^s + \hat{T}_n^\alpha c^{s/2} \right)} 
    = \Erw\brk{\max_{\alpha} \left( \alpha c^s + T_n^\alpha c^{s/2} \right)} + o_n(1).
\end{align*}

To this end, it suffices to verify that
\begin{align*}
    \Erw\brk{\max_{\alpha} \abs{ T_n^\alpha c^{s/2} - \hat{T}_n^\alpha c^{s/2} }} = o_n(1),
\end{align*}
or equivalently
\begin{align*}
    \Erw\brk{ \max_{\sigma \in \{-1,1\}^n} n^{-(r-1)/2} \abs{ U_n(\sigma) - \hat{U}_n(\sigma) } } = o(n).
\end{align*}

We have
\begin{align*}
    U_n(\sigma) - \hat{U}_n(\sigma) 
    = \sqrt{r!} \sum_{\substack{1 \leq v_1 \leq \cdots \leq v_r \leq n \\ \exists\, 1 \leq i < j \leq r: v_i = v_j}} 
    g(v_1, \ldots, v_r) \, J_{v_1, \ldots, v_r} \, f(\sigma_{v_1}, \ldots, \sigma_{v_r}).
\end{align*}

Hence, \( U_n(\sigma) - \hat{U}_n(\sigma) \) is a centered Gaussian random variable with variance \( O(n^{r-1}) \).  
\ignore{\yd{maybe we can move this later since it's used at the end}\dg{good suggestion. 
Moved it here where it logically belongs}}
We recall  the following well-known fact.
For any \( x > 0 \),
\begin{align}\label{eq:Gauss-tail}
\frac{\e^{-x^2/2}}{\sqrt{2\pi}} \left( \frac{1}{x} - \frac{1}{x^3} \right) 
\leq \mathbb{P}\left[\mathcal{N}(0,1) \geq x\right] 
\leq \frac{\e^{-x^2/2}}{x \sqrt{2\pi}}.
\end{align}
By (\ref{eq:Gauss-tail}), we obtain a tail bound:
\begin{align*}
    \mathbb{P}\left[ 
        \max_{\sigma \in \{-1,1\}^n} n^{-(r-1)/2} \abs{U_n(\sigma) - \hat{U}_n(\sigma)} \geq x 
    \right] 
    &\leq 2^n \cdot \mathbb{P}\left[ n^{-(r-1)/2} \abs{U_n(\sigma) - \hat{U}_n(\sigma)} \geq x \right] \\
    &= \exp\left( -O(x^2) + n \log 2 \right).
\end{align*}

Summing over all \( x > n^{2/3} \), we obtain
\begin{align*}
    \sum_{x \geq n^{2/3},~x\in\ZZ} 
    \mathbb{P}\left[ 
        \max_{\sigma \in \{-1,1\}^n} n^{-(r-1)/2} \abs{U_n(\sigma) - \hat{U}_n(\sigma)} \geq x 
    \right] 
    = o_n(1).
\end{align*}

Therefore,
\begin{align*}
    &\Erw\brk{ 
        \max_{\sigma \in \{-1,1\}^n} 
        n^{-(r-1)/2} \abs{U_n(\sigma) - \hat{U}_n(\sigma)} 
    } \\
    &\leq \ceil{n^{2/3}} 
    + \sum_{x \geq n^{2/3},x\in\ZZ} 
    \mathbb{P}\left[ 
        \max_{\sigma \in \{-1,1\}^n} 
        n^{-(r-1)/2} \abs{U_n(\sigma) - \hat{U}_n(\sigma)} \geq x 
    \right] 
    = o(n),
\end{align*}
as claimed.
\end{proof}
\section{Extremizing Gaussians}\label{sec-maximizer}
In light of \Cref{pro:H-V}, our next goal is obtaining asymptotics of \( V_n \). To address this, we first study the value of \( \alpha \) that maximizes \( \alpha c^s + T_n^\alpha c^{s/2} \). For simplicity of notation, let \( d = c^s \).

Let \( \bar{\sigma} = \frac{1}{n} \sum_{i \in [n]} \sigma_i\in[-1,1] \) denote the empirical average of the spin configuration \( \sigma \in \{-1,1\}^n \). Next we express $\alpha$ in terms of $\bar \sigma$.
We have
\begin{align}\label{eq-sum-f-p-square}
\sum_{v_1, \ldots, v_r \in [n]} f(\sigma_{v_1}, \ldots, \sigma_{v_r}) 
&= \sum_{v_1, \ldots, v_r \in [n]} \sum_{m=1}^{\lfloor r/2 \rfloor} \sum_{\substack{R \subset [r] \\ |R| = 2m}} \prod_{j \in R} \sigma_{v_j} \\
&= \sum_{m=1}^{\lfloor r/2 \rfloor} n^r \binom{r}{2m} \left( \frac{1}{n} \sum_{i \in [n]} \sigma_i \right)^{2m} \notag \\
&= n^r \cdot \frac{(1 + \bar{\sigma})^r + (1 - \bar{\sigma})^r - 2}{2}. \notag
\end{align}

Note that $\kappa_1 = -\frac{1}{\aut(F)\cdot 2^{r-1}}$. Combining this with \eqref{eq-sum-f-p-square}, we obtain
\begin{align}\label{eq-ub-alp}
\alpha = \kappa_1 \cdot \frac{(1 + \bar{\sigma})^r + (1 - \bar{\sigma})^r - 2}{2} \leq 0.
\end{align}

\begin{lemma}\label{cla-alp}
    Whp as $n\to\infty$, the maximizer  $\alpha_{\max}$ that maximizes $\alpha d+T_n^\alpha \sqrt{d}$  satisfies 
\begin{align}
    \alpha_{\max}=O(1/\sqrt{d}).
\end{align}
\end{lemma}
\begin{proof}
Consider the covariance of the Gaussian random variables \( \cbc{U_n(\sigma)}_{\sigma\in\cbc{-1,1}^n} \). Note that $(\sum_{m=1}^{\lfloor r/2 \rfloor} \sum_{\substack{R \subset [r], |R| = 2m}} 1)$ counts the number of nonempty even subsets of a set of size $r$, which is bounded from above by $2^r$. Hence, for any \( x, y \in \{-1,1\}^n \), we obtain
\begin{align}\label{eq-cor}
\mathbb{E}[U_n(x) U_n(y)] 
&= r! \sum_{v_1, \ldots, v_r \in [n]} \left( \sum_{m=1}^{\lfloor r/2 \rfloor} \sum_{\substack{R \subset [r] \\ |R| = 2m}} \prod_{j \in R} x_{v_j} \right) \left( \sum_{m=1}^{\lfloor r/2 \rfloor} \sum_{\substack{R \subset [r] \\ |R| = 2m}} \prod_{j \in R} y_{v_j} \right) \\
&\leq r! \cdot 2^{2r} \cdot n^r = C n^r. \notag
\end{align}

Thus, each \( T_n^\alpha \) is the maximum over \( 2^n \) correlated Gaussian variables with variance and covariance bounded above by \( C \kappa_2 n^{-1} \). By a standard argument e.g., combining (\ref{eq:Gauss-tail}) with a union bound, as in the proof of \Cref{pro:H-V}, we obtain that, \whp~,
\[
T_n^\alpha =O(\sqrt{n^{-1}}\sqrt{\log 2^n})= O(1).
\]

Suppose, toward a contradiction, that for sufficiently large \( n \) and \( d = c^s \), there exists a sequence \( t_n \to \infty \) such that for infinitely many $n$, the $\alpha$ achieve the maximum satisfies \( |\alpha| > t_n / \sqrt{d} \). Then from \eqref{eq-ub-alp}, we have \( \alpha < -t_n / \sqrt{d} \), and hence, with high probability, for these infinitely many $n$, as \( n \to \infty \),
\begin{align*}
\max_{\alpha} \left( \alpha d + T_n^\alpha \sqrt{d} \right) 
\leq -t_n \sqrt{d} + O(\sqrt{d}).
\end{align*}

On the other hand, evaluating the expression at \( \alpha = 0 \), we have, with high probability,
\[
\max_{\alpha} \left( \alpha d + T_n^\alpha \sqrt{d} \right)   \geq T_n^0 \sqrt{d}  = O(\sqrt{d}),
\]
which contradicts the previous upper bound. Therefore, no such sequence \( t_n \to \infty \) can exist. The claim in \Cref{cla-alp} follows by the arbitrariness of \( t_n \).
\end{proof}
Hence, it remains to analyze \( T_n^\alpha \) in the regime where \( \alpha \) is close to zero, i.e., $\bar{\sigma}$ is close to zero. However, to the best of our knowledge, there is no known method for obtaining an explicit formula for it.  Instead, we consider the asymptotic behavior as we first let \( d \to \infty \), and then \( r \to \infty \).

\Cref{cla-alp} implies that the value of \( \alpha \) maximizing \( \alpha d + T_n^\alpha \sqrt{d} \) is close to zero when $d$ is large.  
Motivated by this, in this section, we aim to establish upper and lower bounds on  
\[
\max_{\sigma \in \cbc{-1,1}^n} U_n(\sigma),
\]
under the constraint  \( \sum_{i \in [n]} \sigma_i \in [-n h_0, n h_0] \) for some small \( h_0 > 0 \).

In light of (\ref{thm-convergence})
we may assume without the loss of generality
that $n$ is divisible by $4$.
Fix $t>0$ and  \( x, y \in \cbc{-1,1}^n \) such that  
\[
\left| \frac{1}{n} \sum_{i \in [n]} x_i \right|, \quad \left| \frac{1}{n} \sum_{i \in [n]} y_i \right| \leq t.
\]
Fix  \( R_1, R_2 \subset [r] \) satisfying  \( 1 \in R_1 \setminus R_2 \). Then,
\[
\sum_{v_1, \ldots, v_r \in [n]} \prod_{j \in R_1} x_{v_j} \prod_{j \in R_2} y_{v_j} 
= \sum_{v_1 \in [n]} x_{v_1} \sum_{v_2, \ldots, v_r \in [n]} \prod_{j \in R_1 \setminus \{1\}} x_{v_j} \prod_{j \in R_2} y_{v_j} 
= O(t n^r).
\]
Consequently, \eqref{eq-cor} gives
\begin{align}
    \Erw\brk{U_n(x) U_n(y)} 
    &= r! \sum_{v_1, \ldots, v_r \in [n]} 
    \left(
        \sum_{m=0}^{\rou{r/2}} \sum_{\substack{R \subset [r]\\ \abs{R} = 2m}} \prod_{j \in R} x_{v_j}
    \right)\left( 
        \sum_{m=1}^{\rou{r/2}} \sum_{\substack{R \subset [r]\\ \abs{R} = 2m}} \prod_{j \in R} y_{v_j}
    \right) \label{eq-mul-u0} \\
    &= r! \sum_{v_1, \ldots, v_r \in [n]} 
        \left( \sum_{m=1}^{\rou{r/2}} \sum_{\substack{R \subset [r]\\ \abs{R} = 2m}} \prod_{j \in R} x_{v_j} y_{v_j} \right)
        + O(t n^r)\nonumber\\
    &= r! n^r \sum_{m=1}^{\rou{r/2}} \binom{r}{2m} 
        \left( \frac{1}{n} \sum_{i \in [n]} x_i y_i \right)^{2m} 
        + O(t n^r). \nonumber
\end{align}

Let \( o(x,y) = \frac{1}{n} \sum_{i \in [n]} x_i y_i \), then 
\begin{align}\label{eq-mul-u-l}
    \Erw\brk{U_n(x) U_n(y)} 
    &= r! n^r \sum_{m=1}^{\rou{r/2}} \binom{r}{2m} o(x,y)^{2m} + O(t n^r) \\
    &= r! n^r \cdot \frac{(1+o(x,y))^r + (1 - o(x,y))^r - 2}{2} + O(t n^r). \nonumber
\end{align}

Specifically, if \( \sum_{i \in [n]} x_i = \sum_{i \in [n]} y_i = 0 \), then $t=0$, and
\begin{align}\label{eq-cov-sum0}
\Erw\brk{U_n(x) U_n(y)} = r! n^r \cdot \frac{(1+o(x,y))^r + (1 - o(x,y))^r - 2}{2}.    
\end{align}

If in addition \( x = y \), then \( o(x,y) = 1 \), and
\[
\Erw\brk{U_n(x)^2} = \Erw\brk{U_n(x) U_n(y)} = r! n^r (2^{r-1} - 1).
\]


To study $\max_{\sigma \in \cbc{-1,1}^n} U_n(\sigma)$, we use the following fact about the joint tail probability of a bivariate standard normal distribution, 
known also as Slepian bound~\cite{slepian1962one}:

\begin{lemma}\label{lem-cor-normal-tail}
Let \( (Z, Z_{\rho}) \) be a bivariate normal vector with \( Z, Z_{\rho} \sim \mathcal{N}(0,1) \) and correlation \( \mathbb{E}[Z Z_{\rho}] = \rho \in (-1, 1) \). Then, for any \( u > 0 \),
\[
\mathbb{P}(Z > u, Z_\rho > u) 
\leq \frac{(1 + \rho)^2}{2 \pi u^2 \sqrt{1 - \rho^2}}  
\exp \left( -\frac{u^2}{1 + \rho} \right).
\]
\end{lemma}

For normalization, define
\[
W_n(\sigma) = \frac{U_n(\sigma)}{\sqrt{r!} \cdot n^{(r+1)/2} \cdot \sqrt{2^{r-1} - 1}}.
\]
Then, by \eqref{eq-cov-sum0} with \( o(x,y) = 1 \), we have \( W_n(\sigma) \sim \mathcal{N}(0, n^{-1}) \) when $\sum_{i\in[n]}\sigma_i=0$.

Let \( S_j = \cbc{\sigma \in \cbc{-1,1}^n : \sum_{i \in [n]} \sigma_i = j} \). From \eqref{eq-mul-u0}, we observe that the variance of \( W_n(\sigma) \) is the same for all \( \sigma \in S_{nh} \). 
From \eqref{eq-mul-u-l} with \( o(x,y) = 1 \), we further  observe that \( W_n(\sigma) \) for \( \sigma \in S_{nh} \) follows a normal distribution \( \mathcal{N}(0, (1 + f(n,h))/n) \), where
\[
\lim_{h \to 0} \limsup_{n \to \infty} \abs{f(n,h)} = 0.
\]

With the notion above, we next  derive bounds for 
\( \max_{\sigma \in S_0} W_n(\sigma) \) and 
\( \max_{\sigma \in \cup_{\abs{h} \leq h_0} S_{nh}} W_n(\sigma) \):
\begin{lemma}\label{lem-bound-max-normal}
    For any \( \varepsilon > 0 \), there exists \( h(\varepsilon) \in \mathbb{N} \) such that for \( h_0 < h(\varepsilon) \), and \( 4 \mid n \), we have
    \[
    \mathbb{P} \left(
        (1 - \varepsilon) \sqrt{2 \log 2} \leq \max_{\sigma \in S_0} W_n(\sigma) 
        \leq \max_{\sigma \in \cup_{\abs{h} \leq h_0} S_{nh}} W_n(\sigma) 
        \leq (1 + \varepsilon) \sqrt{2 \log 2}
    \right) 
    \geq 1 - o_n(1) - o_r(1).
    \]
\end{lemma}

\begin{proof}
In this proof, we apply the first moment-method to establish the upper bound and the second-moment method for the lower bound.

\textbf{Upper bound:}  
By (\ref{eq:Gauss-tail}),  we have:
\begin{align*}
    \PP\bc{\max_{\sigma \in \bigcup_{\abs{h} \leq h_0} S_{nh}} W_n(\sigma) \geq (1+\vep)\sqrt{2 \log 2} }
    &\leq 2^n \sup_{\abs{h} \leq h_0} \PP\bc{\mathcal{N}(0,1) \geq \frac{(1+\vep)\sqrt{2n \log 2}}{\sqrt{1 + f(n,h)}}} \\
    &\leq \sqrt{\frac{1 + f(n,h)}{4\pi (1+\vep)^2 n \log 2}} \cdot 2^{n\left(1 - \frac{(1+\vep)^2}{1 + f(n,h_0)}\right)}.
\end{align*}

Now choose \( h(\vep) > 0 \) such that
\[
\sup_{\abs{h} \leq h(\vep)} \limsup_{n \to \infty} \abs{f(n,h)} < \vep.
\]
Then it follows that
\begin{align}\label{eq-up-w}
  \PP\bc{\max_{\sigma \in \bigcup_{\abs{h} \leq h_0} S_{nh}} W_n(\sigma) \geq (1+\vep)\sqrt{2 \log 2} } = o_n(1).
\end{align}

\textbf{Lower bound.} To establish the lower bound, we apply the second moment method.  
Consider the set
\begin{align*}
  X = \cbc{ \sigma \in S_0 :~ W_n(\sigma) \geq (1 - \vep) \sqrt{2 \log 2} }.
\end{align*}
We  use the  Stirling’s approximation
\begin{align}\label{eq-stirling}
    m! = (1 + o_m(1)) \sqrt{2\pi m} \left( \frac{m}{e} \right)^m.
\end{align}
In particular, for any \( \beta \in (0, 1) \),
\begin{align}\label{eq-stirling-bin}
    \binom{m}{\beta m} 
    = (1 + o_m(1)) (2\pi m \beta(1 - \beta))^{-1/2} 
    \exp\left( -m \left( \beta \log \beta + (1 - \beta) \log (1 - \beta) \right) \right).
\end{align}

For any \( \sigma \in S_0 \), the variable \( W_n(\sigma) \) follows \( \mathcal{N}(0, 1/n) \).  
By (\ref{eq:Gauss-tail}) and \eqref{eq-stirling-bin}, we compute the first moment of \( |X| \) as
\begin{align}\label{eq-fir-moment-x}
    \Erw\brk{\abs{X}} 
    &= \binom{n}{n/2} \cdot \PP\left( W_n(\sigma) \geq (1 - \vep) \sqrt{2 \log 2} \,\middle|\, \sum_{i \in [n]} \sigma_i = 0 \right) \\
    &= (1 + o_n(1)) n^{-1} \sqrt{ \frac{1}{2 \pi^2 (1 - \vep)^2 \log 2} } \cdot 2^{n(1 - (1 - \vep)^2)}. \nonumber
\end{align}

Fix \( \eta \in (0, 1) \). We next decompose the second moment of \( |X| \) as 
\begin{align}\label{eq-sec-moment-x}
    \Erw\brk{ |X|^2 }
    &= \sum_{\sigma \in S_0} \sum_{\sigma' \in S_0} 
        \PP\left( W_n(\sigma), W_n(\sigma') \geq (1 - \vep) \sqrt{2 \log 2} \right) \nonumber \\
    &\leq \sum_{\sigma \in S_0} \sum_{\substack{ \sigma' \in S_0 \\ |\sigma \cdot \sigma'| \geq \eta n }} 
        \PP\left[ W_n(\sigma) \geq (1 - \vep) \sqrt{2 \log 2} \right] \\
    &\quad + \sum_{\sigma \in S_0} \sum_{\substack{ \sigma' \in S_0 \\ |\sigma \cdot \sigma'| \leq \eta n }} 
        \PP\left[ W_n(\sigma), W_n(\sigma') \geq (1 - \vep) \sqrt{2 \log 2} \right]. \nonumber
\end{align}
For \( \sigma, \sigma' \in S_0 \), note from $4\mid n$ 
that \( 4 \mid \sum_{i \in [n]} \sigma_i \sigma_i' \). To bound the first term on the right-hand side of \eqref{eq-sec-moment-x} from above, for a given $\sum_{i \in [n]} \sigma_i = 0$, we need to count the number of  \(  \sigma' \) satisfying
\begin{align}\label{eq-inner-gamma}
    \sum_{i \in [n]} \sigma_i' = 0 
    \quad \text{and} \quad 
    \sum_{i \in [n]} \sigma_i \sigma_i' = 4y.
\end{align}
Observe that
\[
\sum_{i \in [n]} \sigma_i \sigma_i' 
= 2 \cdot \sum_{i \in [n]} \indic{ \sigma_i = \sigma_i' } - n.
\]
Thus, we deduce:
\begin{enumerate}
    \item \( \abs{ \cbc{ i \in [n] : (\sigma_i, \sigma_i') = (1, 1) } } 
          = \abs{ \cbc{ i \in [n] : (\sigma_i, \sigma_i') = (-1, -1) } } 
          = \frac{n}{4} + y \).
    \item \( \abs{ \cbc{ i \in [n] : (\sigma_i, \sigma_i') = (1, -1) } } 
          = \abs{ \cbc{ i \in [n] : (\sigma_i, \sigma_i') = (-1, 1) } } 
          = \frac{n}{4} - y \).
\end{enumerate}
consequently, the number of such $\sigma'$ is
\begin{align}\label{eq_number_sigma'}
    \binom{n/2}{n/4 - y}^2.
\end{align}
Hence, we can estimate the first term on the right-hand side of \eqref{eq-sec-moment-x} as
\begin{align}\label{eq-sec-moment-x-t1-1}
    &\sum_{\sigma \in S_0} \sum_{\substack{ \sigma' \in S_0 \\ |\sigma \cdot \sigma'| \geq \eta n }} 
        \PP\left[ W_n(\sigma) \geq (1 - \vep) \sqrt{2 \log 2} \right] \\
    &= \Erw\brk{ |X| } \cdot \sum_{|y| \geq \eta n} \binom{n/2}{n/4 - y}^2 
    \leq \Erw\brk{ |X| } \cdot n \cdot \binom{n/2}{(1/4 - \eta) n}^2 \nonumber \\
    &= (1 + o_n(1)) \Erw\brk{ |X| } 
        \cdot \frac{4}{\pi (1 - 16\eta^2)} 
        \exp\left( 
            -n \left[ \left( \tfrac{1}{2} - 2\eta \right) \log \left( \tfrac{1}{2} - 2\eta \right) 
            + \left( \tfrac{1}{2} + 2\eta \right) \log \left( \tfrac{1}{2} + 2\eta \right) 
            \right] 
        \right). \nonumber
\end{align}

Note that
\[
\lim_{x \downarrow 0} \left( x \log x + (1 - x) \log (1 - x) \right) = 0.
\]
Hence, there exists \( \eta = \eta(\vep) \in (0, 1/4) \) such that
\begin{align*}
    \left( \frac{1}{2} - 2\eta \right) \log \left( \frac{1}{2} - 2\eta \right) 
    + \left( \frac{1}{2} + 2\eta \right) \log \left( \frac{1}{2} + 2\eta \right)
    > -\left(1 - (1 - \vep)^2\right) \log 2.
\end{align*}
For such $\eta$, combining this inequality with \eqref{eq-fir-moment-x} and \eqref{eq-sec-moment-x-t1-1}, we obtain
\begin{align}\label{eq-sec-moment-x-t1-2}
    \sum_{\sigma \in S_0} \sum_{\substack{\sigma' \in S_0 \\ |\sigma \cdot \sigma'| \geq \eta n}} 
    \PP\left[ W(\sigma) \geq (1 - \vep) \sqrt{2 \log 2} \right] 
    = o_n(1) \cdot \Erw\brk{ |X| }^2.
\end{align}

We now turn to the second term on the right-hand side of \eqref{eq-sec-moment-x}.  
For \( \sigma, \sigma' \in S_0 \), let \( \sum_{i \in [n]} \sigma_i \sigma_i' = 4y \), where $y$ is an integer. Then \eqref{eq-mul-u-l} (with $t=0$) implies
\[
\Erw\brk{ W_n(\sigma) W_n(\sigma') } 
= \frac{ (1 + 4y/n)^r + (1 - 4y/n)^r - 2 }{ (2^r - 2) n }.
\]
Define
\[
\rho = \frac{ (1 + 4y/n)^r + (1 - 4y/n)^r - 2 }{ 2^r - 2 }.
\]
Since both \( W(\sigma) \) and \( W(\sigma') \) follow \( \mathcal{N}(0, 1/n) \), \Cref{lem-cor-normal-tail} gives
\begin{align}\label{eq-ub-two-normal}
    \PP\left[ W(\sigma), W(\sigma') \geq (1 - \vep) \sqrt{2 \log 2} \right] 
    \leq \frac{(1 + \rho)^2}{4 \pi (1 - \vep)^2 n \sqrt{1 - \rho^2} \log 2} \cdot 2^{- \frac{2(1 - \vep)^2}{1 + \rho} n }.
\end{align}

Consequently, by \eqref{eq-stirling} and \eqref{eq_number_sigma'}, the number of pairs \( (\sigma, \sigma') \) such that  \( \sigma, \sigma' \in S_0 \)  and \( \sum_{i \in [n]} \sigma_i \sigma_i' = 4y \) is
\begin{align}\label{eq_number-pair}&\quad\binom{n}{n/2}\binom{n/2}{n/4 - y}^2=\frac{n!}{\left( (n/4 + y)! \right)^2 \left( (n/4 - y)! \right)^2 }\\
    &= (1 + o_n(1)) (2\pi n)^{-3/2} 
        \cdot \frac{16}{1 - 16y^2/n^2} \cdot 
        \left( \frac{4}{1 + 4y/n} \right)^{n/2 + 2y} 
        \left( \frac{4}{1 - 4y/n} \right)^{n/2 - 2y}\nn \\
    &= (1 + o_n(1)) (2\pi n)^{-3/2} 
        \cdot \frac{16}{1 - 16y^2/n^2} \cdot 2^{2n}\nn\\
        &\quad\times
         \exp\left( 
            -\frac{n}{2} \left[ 
                \left( 1 + \frac{4y}{n} \right) \log \left( 1 + \frac{4y}{n} \right) 
                + \left( 1 - \frac{4y}{n} \right) \log \left( 1 - \frac{4y}{n} \right) 
            \right]
        \right). \nonumber
\end{align}


On the other hand, for \( K > 0 \), combining \eqref{eq-ub-two-normal} and \eqref{eq_number-pair} yields
\begin{align}\label{eq-sec-mid}
    &\sum_{\sigma \in S_0} \sum_{\substack{\sigma' \in S_0 \\ |\sigma \cdot \sigma'| \in [K \sqrt{n}, \eta n]}} 
    \PP\left[ W(\sigma), W(\sigma') \geq (1 - \vep) \sqrt{2 \log 2} \right] \\
    \leq\;& (1 + o_n(1)) (2\pi n)^{-5/2} \cdot \frac{1}{2 (1 - \vep)^2 \log 2} \cdot 2^{2n (1 - (1 - \vep)^2)}\cdot  \sum_{|y| \in [K\sqrt{n}/4, \eta n/4]} 
    \frac{(1 + \rho)^2}{\sqrt{1 - \rho^2}}\frac{16}{1 - \frac{16 y^2}{n^2}} \nonumber \\
    &\times 
       \exp\left( 
        \frac{n}{2} \left[ 
            \frac{4\rho \log 2}{1 + \rho} 
            - \left(1 + \frac{4y}{n}\right) \log\left(1 + \frac{4y}{n}\right) 
            - \left(1 - \frac{4y}{n}\right) \log\left(1 - \frac{4y}{n}\right) 
        \right]
    \right). \nonumber
\end{align}

Let \( t = \frac{4y}{n} \in [-\eta, \eta] \setminus (-K/\sqrt{n}, K/\sqrt{n}) \). Then,
\begin{align*}
    &\frac{4 \rho \log 2}{1 + \rho} 
    - (1 + t) \log(1 + t) 
    - (1 - t) \log(1 - t) \\
    \leq\;& 4 \log 2 \cdot \frac{(1 + t)^r + (1 - t)^r - 2}{2^r - 2} 
    - (1 + t) \log(1 + t) 
    - (1 - t) \log(1 - t).
\end{align*}

On the other hand,
\begin{align}\label{eq-bound-power-asy}
    &4 \log 2 \cdot \frac{(1 + t)^r + (1 - t)^r - 2}{2^r - 2} 
    - (1 + t) \log(1 + t) 
    - (1 - t) \log(1 - t) \\
    =\;& - \bc{1+o_r(1) }t^2 + o_{1/\abs{t}}(t^2). \nonumber
\end{align}
For sufficiently small $\eta$,  there exists \( D > 0 \) which depends on $r$ only, such that for all \( t \in [-\eta, \eta] \),
\begin{align}\label{eq-bound-power}
    &4 \log 2 \cdot \frac{(1 + t)^r + (1 - t)^r - 2}{2^r - 2} 
    - (1 + t) \log(1 + t) 
    - (1 - t) \log(1 - t)\\ \nn
    \leq& - (D + o_r(1)) t^2.
\end{align}

Let 
\[
\rho_{\max} := \sup_{t \in [-\eta, \eta]} \frac{(1 + t)^r + (1 - t)^r - 2}{2^r - 2}.
\]
Then combining \eqref{eq-sec-mid} and \eqref{eq-bound-power}, we obtain
\begin{align*}
    &\sum_{\sigma \in S_0} \sum_{\substack{ \sigma' \in S_0 \\ |\sigma \cdot \sigma'| \in [K \sqrt{n}, \eta n]}} 
    \PP\left[ W(\sigma), W(\sigma') \geq (1 - \vep) \sqrt{2 \log 2} \right] \\
    \leq\;& (1 + o_n(1)) (2\pi n)^{-5/2} \cdot \frac{1}{2 (1 - \vep)^2 \log 2} 
    \cdot 2^{2n (1 - (1 - \vep)^2)} 
    \cdot \frac{(1 + \rho_{\max})^2}{\sqrt{1 - \rho_{\max}^2}} 
    \cdot \frac{16}{1 - \eta^2} \\
    &\times 2 \sum_{y = K\sqrt{n}/4}^{\infty} 
    \exp\left( -8(D + o_r(1)) \cdot \frac{y^2}{n} \right).
\end{align*}

Meanwhile, the tail sum can be bounded as
\begin{align*}
   \sum_{y = K\sqrt{n}/4}^{\infty} 
   \exp\left( -8(D + o_r(1)) \cdot \frac{y^2}{n} \right)
   &= (1 + o_n(1)) \sqrt{n} 
   \int_{K/4}^\infty \exp\left( -8(D + o_r(1)) x^2 \right) \, \mathrm{d}x \\
   &= o_{K,r}(1) \cdot \sqrt{n}.
\end{align*}

Hence,
\begin{align}\label{eq-sec-2}
    &\sum_{\sigma \in S_0} \sum_{\substack{ \sigma' \in S_0 \\ |\sigma \cdot \sigma'| \in [K \sqrt{n}, \eta n]}} 
    \PP\left[ W(\sigma), W(\sigma') \geq (1 - \vep) \sqrt{2 \log 2} \right] 
    = o_{K,r}(1) \cdot \Erw\brk{ |X| }^2.
\end{align}

Finally, we handle the case where \( |\sigma \cdot \sigma'| \in [-K\sqrt{n}, K \sqrt{n}] \).  
In this regime,  \( t=\frac{4y}{n} \in [-K/\sqrt{n}, K/\sqrt{n}] \), and we have
\begin{align}\label{eq-rho-ksqrtn}
    \rho = \frac{(1 + t)^r + (1 - t)^r - 2}{2^r - 2} = o_n(1).
\end{align}
Then, by \eqref{eq-bound-power-asy} and $t \in [-K/\sqrt{n}, K/\sqrt{n}]$,
\begin{align}\label{eq-bound-power-2}
    &\quad 4 \log 2 \cdot \frac{(1 + t)^r + (1 - t)^r - 2}{2^r - 2}
    - (1 + t) \log(1 + t) - (1 - t) \log(1 - t) \\
    &= - (1 + o_r(1) + o_n(1)) t^2.\nn
\end{align}

Combining \eqref{eq-rho-ksqrtn} and \eqref{eq-bound-power-2} with \eqref{eq-sec-mid}, we obtain 
\begin{align*}
    &\sum_{\sigma \in S_0} \sum_{\sigma' \in S_0,\, |\sigma \cdot \sigma'| \in [K \sqrt{n}]} 
    \PP\left[ W(\sigma), W(\sigma') \geq (1 - \vep) \sqrt{2 \log 2} \right] \\
    \leq\; & (1 + o_n(1)) (2\pi n)^{-5/2} \cdot \frac{1}{2 (1 - \vep)^2 \log 2} \cdot 2^{2n (1 - (1 - \vep)^2)} \\
    &\quad \times \sum_{y = -K \sqrt{n}/4}^{K \sqrt{n}/4} 16 \exp\left( - (1 + o_r(1) + o_n(1)) \cdot \frac{8 y^2}{n} \right).
\end{align*}

We estimate the sum by an integral as
\begin{align*}
    &\quad\sum_{y = -K \sqrt{n}/4}^{K \sqrt{n}/4} \exp\left( - (1 + o_r(1) + o_n(1)) \cdot \frac{8 y^2}{n} \right) \\
    &= (1 + o_n(1)) \sqrt{n} \int_{-K/4}^{K/4} \exp\left( - (1 + o_r(1) + o_n(1)) \cdot 8 z^2 \right) \, \mathrm{d}z \\
    &\leq (1 + o_r(1) + o_n(1)) \cdot \frac{\sqrt{2\pi n}}{4}.
\end{align*}

Hence,
\begin{align}\label{eq-sec-3}
    &\sum_{\sigma \in S_0} \sum_{\sigma' \in S_0,\, |\sigma \cdot \sigma'| \in [K \sqrt{n}]} 
    \PP\left[ W(\sigma), W(\sigma') \geq (1 - \vep) \sqrt{2 \log 2} \right] \\
    \leq\; & (1 + o_n(1) + o_r(1)) \cdot \frac{1}{2 \pi^2 (1 - \vep)^2 n^2 \log 2} \cdot 2^{2n (1 - (1 - \vep)^2)} \nonumber \\
    =\; & (1 + o_n(1) + o_r(1)) \cdot \Erw\brk{ |X| }^2. \nonumber
\end{align}

Combining the estimates from \eqref{eq-sec-moment-x}, \eqref{eq-sec-moment-x-t1-2}, \eqref{eq-sec-2}, and \eqref{eq-sec-3}, we conclude:
\begin{align*}
    \Erw\brk{ |X|^2 } 
    = (1 + o_n(1) + o_r(1) + o_{K, r}(1)) \cdot \Erw\brk{ |X| }^2.
\end{align*}

Since both \( \Erw\brk{ |X|^2 } \) and \( \Erw\brk{ |X| }^2 \) are independent of \( K \), we may let \( K \to \infty \) to conclude:
\begin{align*}
    \Erw\brk{ |X|^2 } = (1 + o_n(1) + o_r(1)) \cdot \Erw\brk{ |X| }^2.
\end{align*}

Then, by the Cauchy–Schwarz inequality,
\begin{align}\label{eq-low-w}
    \PP\left[ \max_{\sigma \in S_0} W_n(\sigma) \geq (1 - \vep) \sqrt{2 \log 2} \right] 
    = \PP\left[ |X| \geq 1 \right] 
    \geq \frac{ \Erw\brk{ |X| }^2 }{ \Erw\brk{ |X|^2 } } 
    = 1 + o_n(1) + o_r(1).
\end{align}

Combining the upper bound in \eqref{eq-up-w} with the lower bound in \eqref{eq-low-w} completes the proof of Lemma~\ref{lem-bound-max-normal}.
\end{proof}

\section{Proof of \Cref{thm-limit}}\label{sec-proof}
Since \Cref{thm-convergence} yields that $        \frac{{\rm MinM}(\G(n,p),F)}{n}
        \to m(F,c)$, to compute $m(F,c)$, we only need to consider the $n$ such that $4\mid n$.
For \( \sigma \in S_{nh} \), by \eqref{eq-ub-alp}, we have
\[
\alpha = \kappa_1 \cdot \frac{(1 + h)^r + (1 - h)^r - 2}{2}.
\]
For $h_0$ as defined in \Cref{lem-bound-max-normal},
\[
\alpha_0 = \kappa_1 \cdot \frac{(1 + h_0)^r + (1 - h_0)^r - 2}{2} \leq 0.
\]

Recall from \eqref{def-T-alp} that
\begin{align*}
    T_n^\alpha 
    &= n^{-(r+1)/2} \cdot \sqrt{\kappa_2} \cdot \max_{\sigma \in \cbc{-1,1}^n} U_n(\sigma), \\
    &\quad \text{subject to } 
    n^{-r} \sum_{v_1, \ldots, v_r \in [n]} \kappa_1 f(\sigma_{v_1}, \ldots, \sigma_{v_r}) = \alpha.
\end{align*}

By \Cref{lem-bound-max-normal}, with probability \( 1 + o_n(1) + o_r(1) \),
\[
(1 - \vep) \cdot \sqrt{2\kappa_2 r!(2^{r-1} - 1) \log 2} 
\leq T_n^0 
\leq \max_{\alpha \in [\alpha_0, 0]} T_n^\alpha 
\leq (1 + \vep) \cdot \sqrt{2\kappa_2 r!(2^{r-1} - 1) \log 2}.
\]

It follows that, with probability \( 1 + o_n(1) + o_r(1) \),
\[
V_n = \Erw\left[ \max_{\alpha} \left( \alpha d + T_n^\alpha \sqrt{d} \right) \right] 
\geq \Erw\left[ T_n^0 \cdot \sqrt{d} \right] 
\geq (1 - \vep) \cdot \sqrt{2\kappa_2 r!(2^{r-1} - 1) d \log 2}.
\]

Moreover, by \Cref{cla-alp}, with high probability,
\[
 \max_{\alpha \in [\alpha_0, 0]} \left( \alpha d + T_n^\alpha \sqrt{d} \right)  
= \max_{\alpha} \left( \alpha d + T_n^\alpha \sqrt{d} \right).
\]

On the other hand, as $\alpha\leq 0$,
\begin{align}\label{ineq-T-U}
    \max_{\alpha} \left( \alpha \sqrt{d} + T_n^\alpha  \right)
    \leq n^{-(r+1)/2} \sqrt{\kappa_2} \cdot \max_{\sigma \in \{-1,1\}^n} U_n(\sigma).
\end{align}

We next investigate the right-hand side of \eqref{ineq-T-U}. By \eqref{eq-mul-u-l}, we have
\begin{align*}
    \Erw\brk{U_n(\sigma)^2} 
    &= r! \cdot n^r \sum_{m=1}^{\rou{r/2}} \binom{r}{2m} 1 + O_n(1)  n^{r-1} \sum_{i \in [n]} \sigma_i   \\
    &\leq f(r)^2 \cdot n^r,
\end{align*}
for some function \( f(r) > 0 \).
Consequently, by~\eqref{eq:Gauss-tail},
\begin{align*}
    \mathbb{P} \left( n^{-r/2} \max_{\sigma \in \{-1,1\}^n} U_n(\sigma) \geq x \right)
    &\leq \sum_{\sigma \in \{-1,1\}^n} \mathbb{P} \left( n^{-r/2} U_n(\sigma) \geq x \right) \\
    &\leq 2^n \cdot \frac{f(r)}{x \sqrt{2\pi}} \cdot \exp\left( -\frac{x^2}{2f(r)^2} \right).
\end{align*}

Therefore,
\begin{align*}
    \sum_{\substack{x \geq 2f(r)n^{1/2} \\ \sqrt{\kappa_2} n^{-1/2} x \in \mathbb{Z}}} 
    \mathbb{P} \left( n^{-r/2} \max_{\sigma \in \{-1,1\}^n} U_n(\sigma) \geq x \right) = o_n(1),
\end{align*}
which, combined with\eqref{ineq-T-U}, implies that
\begin{align*}
    \sum_{\substack{x \geq 2f(r)\sqrt{\kappa_2} \\ x \in \mathbb{Z}}} 
    \mathbb{P} \left( \max_{\alpha} \left( \alpha d + T_n^\alpha \sqrt{d} \right) \geq x \right) = o_n(1).
\end{align*}

Hence, the sequence \( \left\{ \max_{\alpha} \left( \alpha \sqrt{d} + T_n^\alpha  \right) \right\}_{n \geq 1} \) is uniformly integrable.  
As a consequence, with probability \( 1 + o_n(1) + o_r(1) \),
\begin{align*}
    V_n &= \Erw\left[ \max_{\alpha} \left( \alpha d + T_n^\alpha \sqrt{d} \right) \right]= \Erw\left[ \max_{\alpha \in [\alpha_0, 0]} \left( \alpha d + T_n^\alpha \sqrt{d} \right) \right] + o_n(\sqrt{d}) \\
   & 
\leq (1 + \vep) \cdot \sqrt{2\kappa_2 r! (2^{r-1} - 1) d \log 2} + o_n(\sqrt{d}).
\end{align*}

\Cref{pro:H-V} then yields that, with probability \( 1 + o_n(1) + o_r(1) \),
\begin{align*}
    \frac{1}{n} \min_{\sigma \in \cbc{-1,1}^n} H(\sigma) 
    = (1 + o_r(1)) \cdot \sqrt{2\kappa_2 r! (2^{r-1} - 1) d \log 2} 
    + \frac{c^s}{2^{r-1} \aut(F)} 
    + o_{n,d}(\sqrt{d}).
\end{align*}
The  result then follows.
\qed


\begin{acks}[Acknowledgments]
The authors would like to thank Annika Heckel, Noela Müller, Connor Riddlesden and 
Lluís Vena for helpful discussions, especially regarding the proof of  \Cref{thm-convergence}.
\end{acks}
\vspace{1em}


\begin{funding}
The authors are supported  by the National Science Foundation under Grant No. DMS-1928930 while they were in residence at the Simons Laufer Mathematical Sciences Institute in Berkeley, California, during the spring semester. 
The work of the second author is also
supported by National Science 
Foundation grant CISE  2233897.
The work of the third author is also supported by the Netherlands Organisation for Scientific Research (NWO) through the Gravitation NETWORKS grant 024.002.003 and the European Union’s Horizon 2020 research and innovation programme under the Marie Skłodowska-Curie grant agreement No. 945045.
\end{funding}
\DeclareRobustCommand{\Hofstad}[3]{#3}
\DeclareRobustCommand{\Esker}[3]{#3}
\bibliographystyle{imsart-number} 
\bibliography{bibmo}       

\end{document}

%% file: preamble.tex
\newcommand{\invisible}[1]{}

\DeclareSymbolFont{extraup}{U}{zavm}{m}{n}
\DeclareMathSymbol{\varheart}{\mathalpha}{extraup}{86}
\DeclareMathSymbol{\vardiamond}{\mathalpha}{extraup}{87}

\usepackage{amssymb}
\usepackage{dsfont}
\usepackage{graphicx}
\usepackage{bbm}
\usepackage{mathrsfs}

\usepackage{xcolor}
\definecolor{aquamarine}{rgb}{0.5, 1.0, 0.83}
\definecolor{aqua}{rgb}{0.0, 1.0, 1.0}
\usepackage{subcaption}
\usepackage{amsthm}

\graphicspath{{pictures/}}

\usepackage{multirow}

\ifpdf
\else

  \usepackage{nohyperref}
  \usepackage[pageref]{backref}
\fi

\usepackage{enumitem}

\usepackage{cleveref}
\theoremstyle{plain}
\newtheorem{theorem}{Theorem}[section]
\newtheorem{corollary}[theorem]{Corollary}

\newtheorem{lemma}[theorem]{Lemma}

\newtheorem{proposition}[theorem]{Proposition}

\newtheorem{example}[theorem]{Example}
\newtheorem{op}[theorem]{Open problem}
\newtheorem{definition}[theorem]{Definition}

\newcommand{\todotemp}[1]{}

\numberwithin{equation}{section}
\numberwithin{theorem}{section}

\makeatletter

\newcommand{\Rbold} {{\mathbb R}}

\newcommand{\e}{{\mathrm e}}

\newcommand\1{\mathbbm{1}}
\newcommand{\indic}[1]{\1_{\{#1\}}}

\newcommand{\Var}{\mathrm{Var}}

\newcommand {\whp} {{whp}}

\newcommand{\R}{\Rbold}

\newcommand{\sss}   { \scriptscriptstyle }

\newcommand{\rO}{{\mathrm{\scriptstyle{O}}}}
\newcommand{\rY}{{\mathrm{\scriptstyle{Y}}}}

\newcommand {\vep}{\varepsilon}

\def\eqalign#1\enalign{
    \begin{align}#1\end{align}
    }
\newcommand{\eq}{\begin{equation}}
\newcommand{\en}{\end{equation}}
\newcommand{\ben}{\begin{enumerate}}
\newcommand{\een}{\end{enumerate}}

\newcommand{\nn}{\nonumber}

\renewcommand{\to}{\rightarrow}

\newcommand{\beginsol}[1]{}